\documentclass[11pt]{amsart}
\usepackage{amsfonts, amssymb, amsmath,mathrsfs}

\usepackage{mathtools}

\usepackage{tikz-cd}
\usetikzlibrary{arrows}

\usepackage{url}
\usepackage{bbm}
\usepackage[hidelinks]{hyperref}
\usepackage[utf8]{inputenc}

\usepackage[english]{babel}

\usepackage[left=25mm,right=25mm,top=38mm,bottom=34mm]{geometry}
\usepackage{url}

\newcommand{\spa}{\medskip}

\newtheorem{thm}[subsection]{Theorem}
\newtheorem{lem}[subsection]{Lemma}
\newtheorem{cor}[subsection]{Corollary}
\newtheorem{prop}[subsection]{Proposition}

\theoremstyle{definition}
\newtheorem{defn}[subsection]{Definition}

\newtheorem{rmk}[subsection]{Remark}

\numberwithin{equation}{section}

\newcommand{\Qp}{\mathbb{Q}_p}

\newcommand{\ga}[2]{\begin{gather}\label{#1}#2 \end{gather}}

\newcommand{\mc}{\mathcal}
\newcommand{\mf}{\mathbf}

\newcommand{\Hom}{{\rm Hom}}
\newcommand{\End}{{\rm End}}
\newcommand{\Ext}{{\rm Ext}}

\newcommand{\Spec}{{\rm Spec \,}}

\newcommand{\sE}{{\mathcal E}}

\newcommand{\sH}{{\mathcal H}}

\newcommand{\sM}{{\mathcal M}}
\newcommand{\sN}{{\mathcal N}}
\newcommand{\sO}{{\mathcal O}}

\newcommand{\sT}{{\mathcal T}}

\newcommand{\sV}{{\mathcal V}}
\newcommand{\sW}{{\mathcal W}}

\newcommand{\C}{{\mathbb C}}

\newcommand{\F}{{\mathbb F}}

\newcommand{\N}{{\mathbb N}}

\newcommand{\Q}{{\mathbb Q}}

\newcommand{\Z}{{\mathbb Z}}

\newcommand{\Vect}{{\mf{Vec}}}

\newcommand{\LS}{{\mf{LS}}}

\newcommand{\Isoc}{\mathbf{Isoc}}

\newcommand{\oFable}{\Isoc^{\dag}_F}

\newcommand{\oi}{\Isoc^{\dag}}

\newcommand{\Perv}{{\mf{Perv}}}
\newcommand{\MHM}{{\mf{MHM}}}
\newcommand{\MHS}{{\mf{MHS}}}

\newcommand{\id}{{\rm id\hspace{.1ex}}}
\newcommand{\Ind}{{\rm Ind}}

\newcommand{\iso}{\xrightarrow{\sim}}

\begin{document}	
	
	\title[Universal extensions]{On the universal extensions  in  Tannakian categories}
	\date{\today}
	\author{Marco D'Addezio and Hélène Esnault}
	\address{Max-Planck-Institut für Mathematik, Vivatsgasse 7, 53111, Bonn, Germany}
	\email{daddezio@mpim-bonn.mpg.de}
	\address{Freie Universit\"at Berlin, Arnimallee 3, 14195, Berlin, Germany}
	\email{esnault@math.fu-berlin.de}

	\begin{abstract} 
		We use the notion of universal extension in a linear abelian category to study extensions of variations of mixed Hodge structure and convergent and overconvergent isocrystals. The results we obtain apply, for example, to prove the exactness of some homotopy sequences for these categories and to study $F$-able isocrystals.
	\end{abstract}
\maketitle
	\tableofcontents
\section{Introduction}

In this note we define and construct \textit{universal extensions} (Definition \ref{defn:univ}) of objects belonging to a linear abelian category. This construction is   convenient to study all the possible extensions of two objects of the category at the same time. On the one hand, a universal extension enjoys more properties than single extensions. On the other hand, every non-trivial extension between two objects embeds into a universal extension (Proposition \ref{univ-ext:p}). Some examples of universal extensions appeared already in the literature. For example, in \cite[§4]{EH06}, universal extensions are used to study the category of flat connections over a smooth variety in characteristic $0$. In our article,  we use universal extensions to solve some open problems on variations of mixed Hodge structures and convergent and overconvergent isocrystals.

\spa Let us present our first application. Let $X$ be a smooth connected complex variety. Write $\MHS$ for the category of graded-polarisable finite-dimensional mixed $\Q$-Hodge structures, $\mf{LS}(X)$ for the category of finite-rank $\Q$-local systems over $X$, and ${\mf M}(X)$ for the category of (graded-polarizable) admissible variations of $\Q$-Hodge structure. Consider the functor $\Psi:{\mf M}(X)\to \mf{LS}(X)$ which associates to a variation its underlying local system.

\begin{thm}[Theorem \ref{vari:t}]\label{i-vari:t}
	For an object $\sV\in \mf{LS}(X)$ the following conditions are equivalent.
	\begin{itemize}
		\item[1)] $\sV$ is a subquotient of a local system of the form $\Psi(\sM)$ for some $\sM \in \mf M(X)$.
		\item[2)] The irreducible subquotients of $\sV$ are subquotients of local systems underlying variations of polarizable pure $\Q$-Hodge structure.
		\item[3)] $\sV$ is a subobject of a local system of the form $\Psi(\sM)$ for some $\sM \in \mf M(X)$.
	\end{itemize}
	
\end{thm}

Besides Deligne's semi-simplicity theorem, the main step in the proof of Theorem \ref{i-vari:t} is to show that for every $\sM,\sN\in \mf M(X)$ a universal extension
$$0\to \Psi(\sM)\to \sE_{\LS(X)}(\Psi(\sN),\Psi(\sM))\to {\rm Ext}^1_{\mf{LS}(X)}(\Psi(\sN), \Psi(\sM))\otimes \Psi(\sN)\to 0$$
in $\mf{LS}(X)$, constructed in Proposition \ref{prop:univ}, comes from an extension \begin{equation}\label{eq}
0\to \sM\to \sE\to V\otimes \sN\to 0
\end{equation}
 in $\mf{M}(X)$, where $V$ is M. Saito's mixed Hodge structure on ${\rm Ext}^1_{\mf{LS}(X)}(\Psi(\sN), \Psi(\sM))$ given by \cite[Thm.~4.3]{Sai90}. We deduce this result from Beilinson's vanishing of the higher extensions of mixed Hodge structures in Lemma \ref{lemma4}. The existence of (\ref{eq}) is a natural enhancement of M. Saito's existence theorem. When $\sM$ and $\sN$ are isomorphic to $\underline{\Q(0)}$, the existence of (\ref{eq}) was already known, as it follows from \cite{HZ87}.

\spa As a consequence of Theorem \ref{i-vari:t}, we deduce an exact sequence for the Tannaka fundamental groups of the categories involved.
\begin{cor}[Corollary \ref{VMHS:c}] \label{i-VMHS:c}
	Let $x$ be a complex point of $X$. The sequence $$\pi_1(\mf{LS}(X),x)\xrightarrow{\Psi^*} \pi_1(\mf M(X),x)\xrightarrow{p_{X*}} 
	\pi_1(\MHS)\to 1$$ is exact. 
\end{cor}
A variant of Corollary \ref{i-VMHS:c} for local systems which can be endowed with a $\Z$-structure is also proved in \cite[Thm. C.11]{And20}.
	
\spa Our second main result is about overconvergent isocrystals, introduced in \cite{Ber} by Berthelot, which are the non-trivial constant-rank coefficients of rigid cohomology. Let $K$ be a complete discretely valued field of mixed characteristic $(0,p)$ and perfect residue field $k$, and let $X$ be a variety over $k$. Write $\oi(X/K)$ for the category of $K$-linear overconvergent isocrystals over $X$. The rigid cohomology groups of $X$ with coefficients in an overconvergent isocrystal might be, in general, infinite-dimensional (see \cite{Cre98}). On the other hand, if the overconvergent isocrystal can be endowed with a Frobenius structure, the cohomology groups are always finite-dimensional, as proven in \cite{Ked06}. Kedlaya's result naturally extends to the full subcategory of \textit{$F$-able} overconvergent isocrystals (Definition \ref{F-able:d}), denoted by $\oi_F(X/K)$. This subcategory is large enough for most of the geometric applications (see for example \cite{Abe18} and \cite{Laz18}). Using universal extensions and Kedlaya's finiteness theorem, we provide a new characterisation of $F$-able overconvergent isocrystals, solving the problem presented in \cite[Rmk. 3.1.9]{Dad}.
\begin{thm}[Theorem \ref{F-able-oc:t}]
	\label{i-F-able-oc:t}
An overconvergent isocrystal over $X$ is $F$-able if and only if it can be embedded into an $F$-invariant overconvergent isocrystal (Definition~\ref{F-able:d}).
\end{thm}

The proof of Theorem \ref{i-F-able-oc:t} can be adapted easily in other situations (see Remark \ref{other-cases:r}). In particular, using universal extensions, one can also obtain similar results for the categories of $\ell$-adic lisse sheaves, $\ell$-adic perverse sheaves, and arithmetic holonomic $\mathcal{D}$-modules.

\spa

 The category of overconvergent isocrystals admits a natural functor $\alpha:\oi(X/K)\to\Isoc(X/K)$ to the category of convergent isocrystals. This functor fails to be fully faithful in general, as proven in \cite[§4]{Abe11}. Combining Theorem \ref{i-F-able-oc:t} and Kedlaya's full faithfulness theorem, proven in \cite{Ked04}, we show that, when $X$ is smooth, $\alpha$ is fully faithful when restricted to the subcategory of $F$-able overconvergent isocrystals. 

\begin{cor}[Corollary \ref{ff-Fable:cor}]\label{i-ff-Fable:cor}If $X$ is a smooth variety, the natural functor $\alpha:\oFable(X/K)\to \Isoc(X/K)$ is fully faithful.
\end{cor}
 Subsequently, in Proposition \ref{coun-ext:p}, we show that Theorem \ref{i-F-able-oc:t} is false on the affine line if one replaces overconvergent isocrystals with convergent isocrystals. Using universal extensions, we prove instead a weaker form of Theorem \ref{i-F-able-oc:t} for convergent isocrystals.

\begin{thm}[Theorem \ref{F-inv-iso:t}]\label{i-F-inv-iso:t}
	If $X$ is a variety over $k$ and $\sM$ is a subquotient of an $F$-invariant convergent isocrystal $\widetilde{\sM}$, then $\sM$ embeds into an $F$-invariant convergent isocrystal in $\langle\widetilde{\sM}\rangle^\otimes$, where $\langle\widetilde{\sM}\rangle^\otimes$ denotes the Tannakian category $\otimes$-generated by $\widetilde{\sM}$ (see Section~\ref{Tan:n}).
\end{thm}

Finally, thanks to Theorem \ref{i-F-able-oc:t} and Theorem \ref{i-F-inv-iso:t}, we recover \cite[Proposition 2.2.4]{AD18} (see Section~\ref{ss:isoc} for the notation).

\begin{cor}[Corollary \ref{conv-oconv-mon:c}]	\label{i-conv-oconv-mon:c}
	Let $X$ be a geometrically connected variety over a finite field $\F_{p^s}$, let $x$ be an $\F_{p^s}$-point of $X$, and let $K$ be a characteristic $0$ complete discretely valued field with residue field $\F_{p^s}$. There exists a functorial commutative diagram
	\begin{center}
		\begin{tikzcd}
			1\arrow{r} & \pi_1(\Isoc(X/K)^{\mathrm{geo}},x)\arrow[r,"\Psi^*"]\arrow[d] & \pi_1(\mf{F^s\textrm{-}}\Isoc(X/K),x)\arrow[r, "p_{X*}"]\arrow[d] & \pi_1(\mf{F^s\textrm{-}}\Isoc(\F_{p^s}))\arrow[r]\arrow[d,"="] &1\\
			1\arrow{r} & \pi_1(\oi(X/K)^{\mathrm{geo}},x)\arrow[r,"\Psi^*"] & \pi_1(\mf{F^s\textrm{-}}\oi(X/K),x)\arrow[r,"p_{X*}"] & \pi_1(\mf{F^s\textrm{-}}\Isoc(\F_{p^s}))\arrow{r} & 1
		\end{tikzcd}
	\end{center}
	with exact rows.
\end{cor}

\spa  We now discuss the method used to prove the exactness of the sequences of Tannaka groups. Let $k$ be a field, and $\mf H\xrightarrow{\beta} \mf G
 \xrightarrow{\alpha} \mf K$ 
 be functors between three Tannakian categories which are compatibly neutralized. By Tannaka duality, this datum is equivalent to 
 homorphisms $ K\xrightarrow{f} G\xrightarrow{g} H$ 
  of affine group schemes.  
 As is well know, 
 it is easy to translate in categorial terms that $f$ is a closed immersion or that $g$ is faithfully flat. It is equivalent to saying in the former case that every object of $\mf K$ is a subquotient of an object coming from $\mf G$, in the latter case that $\beta(\mf H)\subseteq \mf G$ is a full subcategory stable under the operation of taking  subobjects, see \cite[Prop.~2.21]{DM82}. Exactness in the middle is a more subtle property, and has been analysed in  
 \cite[Thm. A.1]{EHS07} (also in \cite[Thm. 5.11]{EH06}). 
The most important condition in  \cite[Thm. A.1]{EHS07} is condition (iii)(c), saying that every object of $\mf K$ should be a subobject of an object coming from $\mf G$. For semi-simple objects of $\mf K$, being a subobject or a subquotient of an object coming from $\mf G$ is the same property. In our article, using universal extensions, we are able to go outside the semi-simple range in the situations we consider.

\spa Condition (iii)(c) can also be replaced with an equivalent condition on rank $1$ objects, as explained in Proposition~\ref{a-prop:exseq}. This variant is mainly due to the work of Bialynicki--Birula--Hochschild--Mostow, \cite{BBHM63}. In Appendix~\ref{sec:tannaka}, we take the opportunity to analyse how their notion of\textit{ observable subgroup} interacts with other properties of morphisms of Tannaka groups.

\spa
{\it Acknowledgments.} This note grew out of a letter dated October 12, 2010 of the second author addressed to Alexander Beilinson with the aim of understanding
the exact sequence in the Hodge world. The second author warmly thanks Alexander Beilinson for the rich exchange at the time. The letter circulated for years. It did not use the theory of mixed Hodge modules and contained a gap in one proof, as was later noticed by Riccardo Ferrario. We thank him, together with 
Simon Pepin Lehalleur and Bruno Klingler for discussions on mixed Hodge modules, Michael Brion for initiating the proof of Lemma~\ref{lem:brion}, 
Brad Drew for explaining to us that the ind-completion of a small abelian category has enough injectives, Claude Sabbah for providing the references in \cite{Sai90}, and Nobuo Tsuzuki for suggesting a proof of Lemma \ref{series:l}. The first author would like to thank Johan Commelin for the enlightening discussions about unipotent variations of mixed Hodge structure during his visit in Freiburg. Finally, we thank Tomoyuki Abe, Jo\~ao Pedro P. dos Santos, and Shuddhodan K. Vasudevan for helpful comments on a first version of this article and Yves André for pointing out \cite{And20}.
 
 \medskip
  
 The first author was funded by the Deutsche Forschungsgemeinschaft (DFG, German Research Foundation) under Germany's Excellence Strategy – The Berlin Mathematics
 Research Center MATH+ (EXC-2046/1, project ID: 390685689) and by the Max-Planck Institute for Mathematics, the second by the Institute for Advanced Study in Princeton where this work was started.

\section{Notation}
\subsection{}
Let $\mf C$ be a locally small category (i.e. ${\rm Hom}(A,B)$ is a set for any objects $A,B\in\mf C$). We write $\rm PSh(\mf C)$ for the category of presheaves of $\mf C$, namely the category of contravariant functors from $\mf C$ to $\mf {Set}$. As in \cite[§6]{KS06}, we write $\Ind(\mf C)\subseteq \rm Psh (\mf C)$ for the full subcategory of ind-objects of $\mf C$ and $\iota_{\mf C} : \mf C\hookrightarrow \Ind(\mf C)$ for the fully faithful functor induced by the Yoneda embedding. By a slight abuse of notation, we denote with the same symbol both an object in $\mf C$ and the presheaf it represents in $\Ind(\mf C)$. If $\mf C$ is an abelian category, for a pair of objects $A,B\in \mf C$, we define $\Ext^1_{\mf C}(B,A)$ as the first extension group of $B$ by $A$. If $b:B\to B'$ is a morphism in $\mf C$, we write $b^*:\Ext^1_{\mf C}(B',A)\to \Ext^1_{\mf C}(B,A)$ for the group homomorphism induced by the pull-back of extensions along $b$.

\subsection{}\label{k-lin:n}
Let $k$ be a field and let $\mf C$ be a $k$-linear abelian category. For every pair of objects $A,B\in \mf C$, the group $\Ext^1_{\mf C}(B,A)$ is naturally a $k$-vector space. If $\lambda\in k$ and $\epsilon\in \Ext^1_{\mf C}(B,A)$, the extension class $\lambda. \epsilon$ is defined as the push-out of $\epsilon$ along $\lambda. \id:A\to A$. The category $\Ind(\mf C)$ is naturally a $k$-linear abelian category and $\iota_{\mf C} : \mf C\hookrightarrow \Ind(\mf C)$ is a $k$-linear fully faithful exact functor. Therefore, it induces a $k$-linear injective morphism $\iota_{\mf C*}:\Ext^1_{\mf C}(B,A) \hookrightarrow \Ext^1_{\Ind(\mf C)}(B,A)$.

\subsection{}
We denote by $\Vect_k$ the category of (not necessarily finite-dimensional) $k$-vector spaces. For an object $A\in \mf C$ and a $k$-vector space $V$ we define $V\otimes_k A$ as the presheaf in $\Ind(\mf C)$ which sends $T\in \mf C$ to the set $V\otimes_k\Hom_{\mf C}(T,A)$. For a vector $v\in V$, we denote by $i_v: B\to V\otimes_k B$ the presheaf morphism in $\Ind(\mf C)$ which, when evaluated at $T\in \mf C$, sends $f \in \Hom_{\mf C}(T,B)$ to $v\otimes f\in V\otimes \Hom(T,B)$.

\subsection{}\label{Tan:n}
Let $\mf C$ be a $k$-linear Tannakian category. For every object $A\in\mf C$ we write $\langle A \rangle^{\otimes}$ for the subcategory \textit{$\otimes$-generated} in $\mf C$ by $A$, namely the smallest strictly full abelian $\otimes$-subcategory of $\mf C$ containing $A$ and closed under duals and the operation of taking subquotients. We write $H^0_{\mf C}(A)$ for the $k$-vector space $\Hom_{\mf C}(\mathbbm{1}_{\mf C},A)$, where $\mathbbm{1}_{\mf C}$ is the unit-object of $\mf C$. Further, a \textit{trivial object} of $\mf C$, for us, is an object of $\mf C$ which is isomorphic to $\mathbbm{1}_{\mf C}^{\oplus n}$ for some $n\in \Z_{\geq 0}$. If $k'/k$ is a field extension we denote by $\mf C_{k'}$ the scalar extension of $\mf C$ from $k$ to $k'$, \cite[Definition 1.2.4]{Sta08}. The category $\mf C_{k'}$ is endowed with an exact faithful $\otimes$-functor $\mf C\to \mf C_{k'}$ mapping $A\mapsto k'\otimes_k A$, which preserves the dimensions of the Hom groups [\textit{ibid.}, Thm. 1.3.18]. This implies that for every $A,B\in \mf C$, there is a natural isomorphism $k'\otimes_k\Ext^1_{\mf C}(B,A)\iso \Ext^1_{\mf C_{k'}}(k'\otimes_k B,k'\otimes_k A)$. 

\subsection{}\label{tan-fun:ss}Let $\mf C$ and $\mf D$ be $k$-linear Tannakian categories. A functor $\alpha:\mf C\to \mf D$, for us, is a \textit{functor of Tannakian categories} if it is an exact $k$-linear $\otimes$-functor. Note that by \cite[Corollaire 2.10]{Del90}, this functor is faithful as well. We denote by $\alpha(\mf C)$ the essential image of $\alpha$. Finally, if $r$ is a positive integer and $V$ is an object of a Tannakian category $\mf C$, we write $\wedge^{r} V$ for the image of $\Sigma_{\sigma\in S_r}(-1)^{{\rm sgn}(\sigma)} \sigma: V^{\otimes r}\to V^{\otimes r}$. Note that $\alpha(\wedge^r V)=\wedge^r \alpha (V)$ because $\alpha$ is exact.

\section{Universal extensions}

  \begin{defn}[Universal extension class] \label{defn:univ}
 	Let $\mf C$ be a $k$-linear abelian category and let $A$ and $B$ be objects in $\mf C$. We say that a class $\epsilon_{\mf C}(B,A)\in \Ext^1_{\Ind(\mf C)}(\Ext^1_{\mf C}(B,A)\otimes B,A)$ is a \textit{universal extension class} of $B$ by $A$ if for every $v\in \Ext^1_{\mf C}(B,A)$, the class $i_{v}^*(\epsilon_{\mf C}(B,A))\in\Ext^1_{\Ind(\mf C)}(B,A)$ is equal to $\iota_{\mf C*}(v)$ (cf. Section~\ref{k-lin:n}). If $$0\to A\to \mc E_{\mf C}(B,A)\to \Ext^1_{\mf C}(B,A) \otimes B\to 0$$ is an exact sequence representing $\epsilon_{\mf C}(B,A)$, we say that $\mc E_{\mf C}(B,A)$ is a \textit{universal extension} of $B$ by $A$.
\end{defn}
We want to prove first that for every pair of objects there exists a unique universal extension up to non-canonical isomorphism. For this, we first need the following elementary lemma.
\begin{lem}\label{ext-iso:l}
	Let $A$ and $B$ be objects of $\mf C$ and let $V$ be a $k$-vector space. There exists an isomorphism of $k$-vector spaces $\Phi:\Ext^1_{\Ind(\mf C)}(V\otimes B,A)\iso \Hom_{\Vect_k}(V, \Ext^1_{\Ind(\mf C)}(B,A))$ defined by $\Phi(\epsilon)(v):=i_v^*(\epsilon)$ for every $\epsilon \in \Ext^1_{\Ind(\mf C)}(V\otimes B,A)$ and $v\in V$. 
\end{lem} 

\begin{proof} Given $\epsilon\in \Ext^1_{\Ind(\mf C)}(V\otimes B,A)$, the map 
	\ga{}{ V\to \Ext^1_{\Ind(\mf C)}(B,A), \ v \mapsto i_v^*(\epsilon) \notag}
	is $k$-linear. This shows that $\Phi$ is well-defined.
	To prove that $\Phi$ is a $k$-linear isomorphism, we choose a basis $I\subseteq V$. In coordinates, $\Phi$ becomes the natural isomorphism $$\Ext^1_{\Ind(\mf C)}(\bigoplus_{i\in I}B,A)\iso\prod_{i\in I} \Ext^1_{\Ind(\mf C)}(B,A).$$
\end{proof}

\begin{prop} \label{prop:univ}
Let $\mf C$ be a $k$-linear abelian category. For every $A,B\in \mf C$ there exists a unique universal extension of $B$ by $A$ up to (non-canonical) isomorphism.
\end{prop}
\begin{proof}
We apply Lemma~\ref{ext-iso:l} to $V=\Ext^1_{\mf C}(B,A)$, obtaining an isomorphism $$\Phi:\Ext^1_{\Ind(\mf C)}(\Ext^1_{\mf C}(B,A)\otimes B,A)\iso \Hom_{\Vect_k}(\Ext^1_{\mf C}(B,A), \Ext^1_{\Ind(\mf C)}(B,A))$$ such that $\Phi(\epsilon)(v)=i_v^*(\epsilon)$ for every $\epsilon \in \Ext^1_{\Ind(\mf C)}(V\otimes B,A)$ and $v\in V$.
By construction, $\epsilon$ is a universal extension class if and only if $\Phi(\epsilon)=\iota_{\mf C*}$ as an element in $\Hom_{\Vect_k}(\Ext^1_{\mf C}(B,A), \Ext^1_{\Ind(\mf C)}(B,A))$. Since $\Phi$ is an isomorphism, $\Phi^{-1}( \iota_{\mf C*})$ is the unique universal extension class of $B$ by $A$. This concludes the proof.
\end{proof}

The next proposition is the main result on universal extensions we shall use in the rest of the article. 

\begin{prop}\label{univ-ext:p}  Let $\mf C$ be a $k$-linear abelian category and	$0\to A\xrightarrow{a} E\xrightarrow{b} B\to 0$ be an exact sequence in $\mf C$. If $A$ is a subobject of $A'\in \mf C$ and $B$ is a direct summand of $B'\in\mf C$, then $E$ is a subobject of $\mc E_{\mf C}(B',A')\oplus B'$, where $\mc E_{\mf C}(B',A')$ is any universal extension of $B'$ by $A'$.  
\end{prop}
\begin{proof}
	Let $c$ be a section of the quotient $B'\twoheadrightarrow B'/B$. We push-out the exact sequence $0\to A \xrightarrow{(a,0)}E\oplus (B'/B)\xrightarrow{b\oplus c} B'\to 0$ along $A\subseteq A'$, obtaining the following commutative diagram with exact rows
	\begin{center}
		\begin{tikzcd}
		0\arrow[r] &  A \arrow[r,"{(a,0)}"] \arrow[d,hook]\arrow[dr, phantom, "\square"] & E\oplus (B'/B)\arrow[r,"b \oplus c"]\arrow[hook,d] & B'\arrow{r}\arrow[d,"="] &0\\
		0\arrow{r} & A'\arrow{r} &  E'  \arrow{r} & B'\arrow{r} & 0.
		\end{tikzcd}
	\end{center}
As $E$ is a subobject of $E'$,  it is enough to prove the final result for the lower exact sequence. Let $v'\in\Ext^1_{\mf C}(B',A')$ be the class of the extension $E'$ and let $\mc E_{\mf C}(B',A')$ be a universal extension of $B'$ by $A'$. If $v'=0$, then $E'=A'\oplus B'$ is trivially a subobject of $\mc E_{\mf C}(B',A')\oplus B'$. If $v'\neq 0$, the extension $E'$ is obtained from $\mc E_{\mf C}(B',A')$ via a pull-back along the injective morphism $i_{v'}:B'\hookrightarrow
	\Ext^1_{\mf C}(B',A')\otimes B'$. This shows that $E'$ is a subobject of $\mc E_{\mf C}(B',A')$, as we wanted.
	
\end{proof}
We end this section with some results assuring that the universal extensions are in $\mf C$, rather than in $\Ind(\mf C)$.

\begin{lem}\label{ext-dim-univ-ext:l}Let $\mf C$ be a $k$-linear abelian category such that each object is noetherian. For all $A,B\in \mf C$, the Ext group $\Ext^1_{\mf C}(B,A)$ is a finite-dimensional $k$-vector space if and only if one (or equivalently any) universal extension of $B$ by $A$ is in the essential image of $\iota_{\mf C}:\mf C\hookrightarrow \Ind(\mf C)$.
\end{lem}

	\begin{proof}
	Thanks to \cite[Lemme 4.2.1]{Del87}, if $\mf C$ is an abelian category  such that each object is noetherian, the essential image of $\iota_{\mf C}:\mf C \hookrightarrow \Ind(\mf C)$ is the Serre subcategory of noetherian objects. Therefore, the property that a universal extension of $B$ by $A$ is in $\mf C$ is equivalent to the property that $\Ext^1_{\mf C}(B,A)\otimes B$ is in $\mf C$. The last condition is equivalent to   $\Ext^1_{\mf C}(B,A)$ being finite-dimensional.
	\end{proof}

\begin{lem}\label{ext-dim:l}Let $k$ be a characteristic $0$ field and let $\mf C$ be a $k$-linear Tannakian category $\otimes$-generated by one object. For every $A,B\in \mf C$, any universal extension of $B$ by $A$ is in the essential image of $\iota_{\mf C}:\mf C\hookrightarrow \Ind(\mf C)$.
\end{lem}

\begin{proof}Thanks to Lemma \ref{ext-dim-univ-ext:l}, it is enough to prove that the Ext groups of $\mf C$ are finite-dimensional. This latter property is invariant by scalar extension (see Section~\ref{Tan:n}). As $\mf C$ is neutralised over a finite extension of $k$ by \cite[§2.8]{Del90}, we may assume that $\mf C$ is neutral  and  $k$ is algebraically closed. 
By  Tannaka duality, the result then  follows from Lemma~\ref{lem:brion}.
\end{proof}

\begin{lem} \label{lem:brion} Let  $k$ be an algebraically closed field of characteristic $0$. If $V$ is a  finite-dimensional representation of an affine group $G$ of finite type over $k$, the $k$-vector space $H^1(G,V)$ is finite-dimensional.
\end{lem}
\begin{proof}
Let $N\subseteq G$ is the unipotent radical of $G$. By the Lyndon--Hochschild--Serre spectral sequence, the sequence $$0\to H^1(G/N, V^N)\to H^1(G,V)\to H^1(N,V)^{G/N} $$ is exact. The algebraic group $G/N$ is reductive and $k$ has characteristic $0$, thus the inclusion 
$W^{G/N}\hookrightarrow W$ splits for every finite-dimensional representation $W$ of $G/N$. This implies  that $H^1(G/N, V^N)=0$. We are reduced to proving that $H^1(G,V)$ is finite-dimensional when $G$ is a unipotent algebraic group. In this case, $V$ is a finite successive extension of rank $1$ trivial representations, thus we may further assume that $V$ is trivial. We argue by induction on the dimension $d$ of $G$. If $d=1$, the algebraic group $G$ is isomorphic to $\mathbb{G}_a$ and $H^1(\mathbb{G}_a,k)=\Hom_{k\textrm{-}\rm{grp}}(\mathbb{G}_a,\mathbb{G}_a)=k$, \cite[Lemma 4.21]{Jan07}.  
In general, if  $d\ge 1$, there is an exact sequence $1\to N\to G\to \mathbb{G}_a \to 1$ where $N$ is a unipotent group of dimension $(d-1)$. Thanks to Lyndon--Hochschild--Serre spectral sequence, we obtain an exact sequence $$0\to H^1(\mathbb{G}_a,k)\to H^1(G,k)\to H^1(N,k)^{G/N}.$$ The induction hypothesis  concludes the proof.

\end{proof}

\begin{rmk}
	Lemma~\ref{lem:brion} is probably well-known to experts, but we did not find a proof in the literature. We discussed the proof presented here with Michel Brion.  See also the argument after Corollary 35 in \cite{dS11} for a finiteness result in positive characteristic.
\end{rmk}

\section{Variations of mixed Hodge structure}
\subsection{}\label{VMHS:ss}
 The main goal of this section is to prove Theorem \ref{vari:t}. Let us first introduce our setting. For a smooth connected complex variety $p_X: X \to \Spec \C$, we denote by $\MHS$ the category of graded-polarisable mixed $\Q$-Hodge structures, by ${\mf M}(X)$ the category of (graded-polarisable) admissible variations of mixed $\Q$-Hodge structure (cf. \cite[§1.9]{Kas86}) over $X$, and by ${ \mf{LS}}(X)$ the category of $\Q$-local systems over $X$. 
These categories are $\Q$-linear neutral Tannakian categories. We denote by $\Q(0)$, $\underline{\Q(0)}$, and $\underline{\Q}$ the unit objects of  $\MHS$, ${\mf M}(X)$, and $\LS(X)$ respectively. There exists a natural exact $\otimes$-functor $\Psi: {\mf M}(X)\to { \mf{LS}}(X)$ which associates to an object of ${\mf M}(X)$ its underlying $\Q$-local system. Also, we have an inverse image functor $p_X^*: \MHS\to { \mf M}(X)$ which identifies the category of graded-polarizable mixed $\Q$-Hodge structure with the full subcategory of ${ \mf M}(X)$ of objects with trivial underlying local system. For $\sM \in \mf M(X)$, the maximal trivial sublocal system $H^0_{\rm Betti}(X,\Psi(\sM))\otimes \underline{\mathbb{Q}}\subset \Psi(\sM)$ carries a graded-polarizable mixed $\Q$-Hodge structure by \cite[Proposition 4.19]{SZ}. This generalizes Deligne's fixed part theorem \cite[§4.1]{D}. Even in degree $i\geq 1$, the group $H^i_{\rm Betti}(X, \Psi(\sM))$ carries a graded-polarizable mixed Hodge structure
(\cite{Z} over curves and for Gau\ss-Manin $\sM$ and \cite[Thm.~4.3]{Sai90} in general).

\subsection{} \label{ss:saito}
We shall use M. Saito's construction, which we recall now. The category $\mf M(X)$ admits a $\Q$-linear fully faithful exact functor to the category of (graded-polarizable) \textit{mixed Hodge modules} $\MHM(X)$. The category of mixed Hodge modules is endowed with a natural functor $\Psi:\MHM(X)\to \Perv(X)$, where $\Perv(X)$ is the category of perverse sheaves over $X$. The essential image of $\mf M(X)$ in $\MHM(X)$ is the full subcategory of \textit{smooth mixed Hodge modules}, i.e. of  those modules the  underlying perverse sheaf  of which is a (shifted) local system. 
M. Saito constructed a direct image functor $p_{X*}:\mathrm{D^b}(\MHM(X))\to \mathrm{D^b}(\MHS)$ which is compatible with the direct image functor for perverse sheaves, \cite[Thm.~4.3]{Sai90}. Following M. Saito, we write $\sH^ip_{X*}(-)$ for $H^i(p_{X*}(-))$. By the aforementioned compatibility, we have that for $\sM \in \mf{M}(X)$, the vector space underlying the mixed Hodge structure of $\sH^ip_{X*}(\sM)$ is $H^i_{\rm Betti}(X, \Psi(\sM))$. We shall need the following lemma.

	\begin{lem}\label{lemma4}
	For every $\sM \in \mf M(X)$, the image of the morphism $${\rm Ext}^1_{\mf M(X)}(	\underline{\Q(0)},\sM)\to {\rm Ext}^1_{\mf{LS}(X)}(\underline	\Q, \Psi(\sM))=H^1_{\rm Betti}(X, \Psi(\sM))$$ induced by $\Psi$ is $H^0_{\MHS}(\sH^1p_{X*}(\sM))$.
\end{lem}
\begin{proof}

	We have a commutative diagram of functors
	\begin{center}
		\begin{tikzpicture}[scale=2.0]
			\node (A) at (0,0) {${\rm Ind}(\MHS)$};
			\node (C) at (-1.5,0.75) {${\rm Ind}(\MHM(X))$};
			\node (D) at (1.5,0.75) {$\Vect_\Q$};
			\path[->,font=\scriptsize,>=angle 90]
			(C) edge node[below left]{$\widetilde{\sH}^0p_{X*}$} (A)
			(A) edge node[below right]{$H^0_{\Ind(\MHS)}(-)$} (D)
			(C) edge node[above]{$H^0_{\Ind(\MHM(X))}(-)$} (D);
		\end{tikzpicture}
	\end{center}
	where $\widetilde{\sH}^0p_{X*}:{\rm Ind}(\MHM(X))\to {\rm Ind}(\MHS)$ is the left Kan extension of the functor $\sH^0p_{X*}$ of Section~\ref{ss:saito}. Note that $$H^0_{\Ind(\MHM(X))}(-)=\Hom_{\Ind(\MHM(X))}(\underline{\Q(0)},-)\textrm{ and }H^0_{\Ind(\MHS)}(-)=\Hom_{\Ind(\MHS)}({\Q(0)},-).$$ The categories ${\rm Ind}(\MHM(X))$ and ${\rm Ind}(\MHS)$ have enough injectives by \cite[Prop. 1.7]{Hub93} and, by [\textit{ibid.}, Thm. 2.6], the restriction of the right derived functor $\textbf{R}^i(\widetilde{\sH}^0p_{X*}):{\rm Ind}(\MHM(X))\to {\rm Ind}(\MHS)$ to $\MHM(X)$ coincides with $\sH^ip_{X*}$ for every $i$. Since $\widetilde{\sH}^0p_{X*}$ admits a left-adjoint which is exact, we can construct the Grothendieck spectral sequence associated to the composition of $\widetilde{\sH}^0p_{X*}$ and $H^0_{\Ind(\MHS)}(-)$. We get for every $\sM\in \MHM(X)$ an exact sequence of $\Q$-vector spaces
	\ga{}{ {\rm Ext}^1 _{\Ind(\MHM(X))}( \underline{\Q(0)}, \sM)  \to H^0_{\Ind(\MHS)}(\sH^1p_{X*}(\sM)) \to  {\rm Ext}^2_{\Ind(\MHS)}(\sH^0p_{X*}(\sM)). \notag}
	By [\textit{ibid.}, Prop. 2.2], the previous exact sequence can be rewritten as
	$${\rm Ext}^1 _{\MHM(X)}( \underline{\Q(0)}, \sM)  \to H^0_{\MHS}(\sH^1p_{X*}(\sM)) \to  {\rm Ext}^2_{\MHS}(\Q(0),\sH^0p_{X*}(\sM)).$$
	Thanks to \cite[§1.10]{B}, ${\rm Ext}^2_{\MHS}(\Q(0),\sH^0p_{X*}(\sM))=0, $which implies that the morphism $${\rm Ext}^1 _{\MHM(X)}(\underline{\Q(0)}, \sM)  \to H^0_{\MHS}(\sH^1p_{X*}(\sM))$$ is surjective.
	
	To prove that ${\rm Ext}^1 _{\mf M(X)}(\underline{\Q(0)}, \sM)  \to H^0_{\MHS}(\sH^1p_{X*}(\sM))$ is surjective as well, we note that in $\Perv (X)$, an extension of two (shifted) local systems is again a (shifted) local system. Therefore, every extension of $\underline{\Q(0)}$ by $\sM$ in $\MHM(X)$ is realised in $\mf M(X)$. This finishes the proof.
	
\end{proof}

We are ready now to prove our main result on local systems.

\begin{thm}\label{vari:t}
	For an object $\sV\in \mf{LS}(X)$ the following properties are equivalent.
	\begin{itemize}
		\item[1)] $\sV$ is a subquotient of a local system of the form $\Psi(\sM)$ for some $\sM \in \mf M(X)$.
		\item[2)] The irreducible subquotients of $\sV$ are subquotients of local systems underlying variations of polarizable pure $\Q$-Hodge structure.
		\item[3)] $\sV$ is a subobject of a local system of the form $\Psi(\sM)$ for some $\sM \in \mf M(X)$.
	\end{itemize}
	
\end{thm}

\begin{proof}It is clear that 1) implies 2) and 3) implies 1). It remains to prove that 2) implies 3). To prove this implication we argue by induction on the length of a Jordan--H\"older filtration for $\sV$. If $\sV$ is irreducible, by Deligne's semi-simplicity theorem, \cite[§4.2]{D}, $\sV$ is a direct summand of a local system underlying a variation of polarisable pure $\Q$-Hodge structure. Suppose now that $\sV$ admits a non-trivial irreducible quotient $\sV\twoheadrightarrow \sW$. By 2) and the previous step, $\sW$ is a direct summand of $\Psi(\sM)$ where $\sM \in \mf M(X)$ is a variation of pure $\Q$-Hodge structure. Also, by induction, there exists $\sN \in \mf M(X)$ such that $\rm Ker(\sV \twoheadrightarrow \sW)$ is a subobject of $\Psi(\sN)$. By Proposition \ref{univ-ext:p}, the local system $\sV$ is a subobject of $\sE_{\LS(X)}(\Psi(\sN),\Psi(\sM))\oplus\Psi(\sN)$ where $\sE_{\LS(X)}(\Psi(\sN),\Psi(\sM))$ is a universal extension of $\Psi(\sN)$ by $\Psi(\sM)$. As $\LS(X)$ is a Tannakian category with finite-dimensional Ext groups, 
	$\sE_{\LS(X)}(\Psi(\sN),\Psi(\sM))$ is an object of $\LS(X)$ 
by Lemma~\ref{ext-dim-univ-ext:l}.  It remains to show that $\sE_{\LS(X)}(\Psi(\sN),\Psi(\sM))$ is in the essential image of $\Psi: \mf M(X)\to \LS(X)$.  To this aim we use Lemma \ref{lemma4}.
\spa

We write $E$ for the mixed Hodge structure $\sH^1p_{X*}(\sN^\vee \otimes \sM)$ on the $\Q$-vector space $$H^1_{\rm Betti}(X,\Psi(\sN)^\vee\otimes\Psi(\sM))={\rm Ext}^1_{\mf{LS}(X)}(\Psi(\sN), \Psi(\sM)).$$ If $\Phi:\Ext^1_{\mf{LS}(X)}(\Psi(E)\otimes \Psi(\sN),\Psi(\sM))\iso \End_{\Vect_k}(\Psi(E))$ is the isomorphism of Lemma \ref{ext-iso:l}, the universal extension class $\epsilon$ of $\Psi(\sN)$ by $\Psi(\sM)$ is such that $\Phi(\epsilon)=\id$. On the other hand, applying Lemma \ref{lemma4} to $p_X^*(E)^\vee \otimes\sN^\vee\otimes \sM  \in \mf M(X)$, we deduce that the morphism 
	$$\Psi_*:{\rm Ext}^1_{\mf M(X)}(\underline{\Q	(0)},p_X^*(E)^\vee \otimes\sN^\vee\otimes \sM )\to H^0_{\MHS}(E^\vee \otimes E)={\rm End}_{\MHS}(E)$$ induced by $\Psi$ is surjective. In particular, the identity $\id \in \End_{\Vect_\Q}(\Psi(E))$ is in the image of $\Psi_*$. Combining these two facts, and exploiting the natural isomorphism
	$${\rm Ext}^1_{\mf M(X)}(p_X^*(E) \otimes\sN,\sM)\iso{\rm Ext}^1_{\mf M(X)}(\underline{\Q	(0)},p_X^*(E)^\vee \otimes\sN^\vee\otimes \sM ),$$ we deduce that the universal extension $$0\to \Psi(\sM)\to \sE_{\LS(X)}(\Psi(\sN),\Psi(\sM))\to {\rm Ext}^1_{\mf{LS}(X)}(\Psi(\sN), \Psi(\sM))\otimes \Psi(\sN)\to 0$$ comes from an extension $$0\to \sM\to \sE\to p_X^*(E)\otimes \sN\to 0$$ in $\mf{M}(X)$. This concludes the proof.

\end{proof}
\begin{rmk}
Similarly, in Theorem \ref{vari:t} one may replace the category $\LS(X)$ with $\Perv(X)$ and $\textbf{M}(X)$ with $\MHM(X)$ and obtain a variant for perverse sheaves.\end{rmk}

\subsection{} \label{ss:neut}
Let $x\in X(\C)$ be a complex point. It neutralises the categories introduced in Section \ref{VMHS:c}. This defines fundamental groups $\pi_1(\mf M(X), x)$ and $\pi_1(\mf{LS}(X),x)$. Recall that the  category $\MHS$ is also naturally neutralised by sending a mixed $\Q$-Hodge structure to its underlying vector space. This defines a fundamental group $\pi_1(\MHS)$. All these fundamental groups are affine group schemes over $\Q$. One has a complex of such 
\begin{equation*}
\pi_1(\mf{LS}(X),x)\xrightarrow{\Psi_*} \pi_1(\mf M(X),x)\xrightarrow{p_{X*}} 
\pi_1(\MHS).
\end{equation*}
As there are $\Q$-local systems which are not subquotients of variations of mixed $\Q$-Hodge structure, $\Psi_*$ is not injective in general. Let $\mf{LS}(X)^{\mathrm{hdg}}$ be the full
subcategory of $\mf{LS}(X)$ consisting of the  objects which are $\Q$-local systems which are subquotients  as local systems of  variations of mixed $\Q$-Hodge structure. The affine group scheme $\pi_1(\mf{LS}(X)^{\mathrm{hdg}}, x)$ is the image of $\pi_1(\mf{LS}(X),x)$ in $\pi_1(\mf M(X), x)$.

\begin{cor} \label{VMHS:c}
 The sequence $$1\to\pi_1(\mf{LS}(X)^{\mathrm{hdg}},x)\xrightarrow{\Psi^*} \pi_1(\mf M(X),x)\xrightarrow{p_{X*}} 
 \pi_1(\MHS)\to 1$$ is exact. 
\end{cor}

\begin{proof}
The functor $\mf M(X)\to\MHS$ which sends an object $\sM \in \mf M(X)$ to its fibre at $x$ induces a splitting of the morphism
$p_{X*}$, thus  $p_{X*}$ is
faithfully flat. To prove the exactness in the middle we apply Proposition \ref{a-prop:exseq}. By Theorem \ref{vari:t}, every object in $\mf{LS}(X)^{\mathrm{hdg}}$ is in fact a subobject of a local system underlying an admissible variation of mixed $\Q$-Hodge structure. This shows that $\Psi$ is observable. To conclude, it is enough to note that, as recalled in Section~\ref{VMHS:ss}, the maximal trivial subobject $H^0_{\rm Betti}(X,\Psi(\sM))\otimes \underline{\mathbb{Q}}\subset \Psi(\sM)$ carries a graded-polarizable mixed $\Q$-Hodge structure. Alternatively, one can prove that $\Psi$ is observable by combining Deligne's semi-simplicity theorem and Lemma \ref{a-suf-obs:l}.

\end{proof}
\begin{rmk}
In Section~2 of {\it {The Hodge theoretic fundamental group and its cohomology}} in The geometry of algebraic cycles \textbf{9} (2010), 3--22, D. Arapura
defines the Tannakian category of \textit{enriched local systems} by the axioms (E1)--(E4) and claims in Theorem 2.6 of {\it loc. cit.} an exact sequence like the one in Corollary \ref{VMHS:c}. Unfortunately, the proof is not correct (the statement “An element of $im$...” is wrong), thus the Hodge theoretic consequences of it (Theorem~3.8 and Theorem~5.2) are not proved. One way to remedy this is to add “Deligne's semi-simplicity theorem” to the list of axioms, say 
\begin{itemize}
\item[(E5):] $\phi$ preserves semi-simplicity.
\end{itemize}
Thanks to Lemma \ref{a-suf-obs:l}, (E5) assures that $\phi$ is observable and then one can apply Proposition \ref{a-prop:exseq}.
\end{rmk}

\section{Isocrystals}
\subsection{}Let $K$ be a complete discretely valued field of mixed characteristic $(0,p)$ and perfect residue field $k$ and let $X$ be a variety over $k$. In this section we study the categories $\Isoc(X/K)$ and $\oi(X/K)$ of $K$-linear convergent and overconvergent isocrystals over $X$ (cf. \cite{Ogu84} and \cite{Ber}). We write $F$ for the absolute Frobenius of $X$. This defines, via pullback, autoequivalences $F^*$ on the two categories of isocrystals (see \cite[Cor. 4.10]{Ogu84} and \cite[Cor. 6.2]{Laz}).
	
\begin{defn}\label{F-able:d}We say that a convergent (resp. overconvergent) isocrystal $\sM$ over $X$ is \textit{$F$-able} if for every irreducible subquotient $\sN$ of $\sM$ there exists $n>0$ such that $(F^*)^n\sN \simeq \sN$. We write $\Isoc_{F}^{(\dag)}(X/K)$ for the category of $K$-linear $F$-able (over)convergent isocrystals. Also, we say that a $K$-linear convergent or overconvergent isocrystal $\sM$ is \textit{$F$-invariant} if $F^*\sM \simeq \sM$.
\end{defn}

Every subobject of an $F$-invariant convergent or overconvergent isocrystal is $F$-able (see for example \cite[Corollary 3.1.8]{Dad}). In this section, we want to prove that the converse is also true for $F$-able  overconvergent isocrystals. We start with the next lemma.

\begin{lem}\label{irr-isoc:l}
	Every semi-simple $F$-able convergent (resp. overconvergent) isocrystal $\sM$ is the direct summand of an $F$-invariant semi-simple convergent (resp. overconvergent) isocrystal $\sN$.
\end{lem}
\begin{proof}
We may assume that $\sM$ irreducible. By the assumption, there exists $n>0$ such that $(F^*)^n\sM \simeq \sM$. The isocrystal $\sM$ is then a direct summand of $\sN:=\bigoplus_{i=0}^{n-1} (F^*)^i\sM$, which is $F$-invariant. Since $F^*$ is an autoequivalence of the category of convergent (resp. overconvergent) isocrystals, every summand $(F^*)^i\sM$ is irreducible. Therefore $\sN$ is semi-simple.
\end{proof}

We are now ready to prove the main theorem of this section.

\begin{thm} 
	\label{F-able-oc:t}
An overconvergent isocrystal over a variety $X$ over $k$ is $F$-able if and only if it can be embedded into an $F$-invariant overconvergent isocrystal.
\end{thm}

\begin{proof}We have already seen that the second condition implies the first one, thus we have to prove the converse. Let $\sM^\dag$ be an $F$-able overconvergent isocrystal. We prove the statement by induction on the length of a Jordan--H\"older filtration of $\sM^\dag$. If $\sM^\dag$ is irreducible, thanks to Lemma \ref{irr-isoc:l}, it is the direct summand of an $F$-invariant overconvergent isocrystal. Suppose now that $\sM^\dag$ is an extension of an irreducible overconvergent isocrystal $\sM^\dag_1$ by an overconvergent isocrystal $\sM^\dag_2$. By the previous step, $\sM^\dag_1$ is a direct summand of an $F$-invariant overconvergent isocrystal $\sN^\dag_1$, while, by the induction assumption, 
 $\sM^\dag_2$ embeds into an $F$-invariant overconvergent isocrystal $\sN^\dag_2$.
 
 \spa By Proposition \ref{univ-ext:p}, the overconvergent isocrystal $\sM^\dag$ is then a subobject of $\sE_{\oFable(X/K)}(\sN^\dag_2,\sN^\dag_1)\oplus \sN^\dag_2$, where $\sE_{\oFable(X/K)}(\sN^\dag_2,\sN^\dag_1)$ is a universal extension of $\sN^\dag_2$ by $\sN^\dag_1$. By \cite{Ked06}, the Ext groups of $\oFable(X/K)$ are finite-dimensional $K$-vector spaces. Thus, by Lemma \ref{ext-dim-univ-ext:l}, the ind-object $\sE_{\oFable(X/K)}(\sN^\dag_2,\sN^\dag_1)$ is actually an object of $\oFable(X/K)$. It remains to show that it is $F$-invariant. This follows from the fact that the Frobenius pull-back $F^*$ induces an autoequivalence on $\oFable(X/K)$, thus $$F^*(\sE_{\oFable(X/K)}(\sN^\dag_2,\sN^\dag_1))=\sE_{\oFable(X/K)}(F^*\sN^\dag_2,F^*\sN^\dag_1)\simeq \sE_{\oFable(X/K)}(\sN^\dag_2,\sN^\dag_1),$$ as we wanted.
\end{proof}

\begin{rmk} \label{other-cases:r}We preferred to present Theorem \ref{F-able-oc:t} for $F$-able overconvergent isocrystals, but the proof works unchanged in various other cases. Indeed, we have simply used the fact that $\oFable(X/K)$ is a $K$-linear abelian category with noetherian objects and finite-dimensional Ext groups and $F^*:\oFable(X/K)\to \oFable(X/K)$ is an autoequivalence. These conditions are satisfied, for example,  by the categories of $\ell$-adic lisse sheaves, $\ell$-adic perverse sheaves, and arithmetic holonomic $\mathcal{D}$-modules.
\end{rmk}

To prove one of the consequences of Theorem \ref{F-able-oc:t}, namely Corollary \ref{ff-Fable:cor} we need another general lemma on Tannakian categories.

\begin{lem}\label{H0:l}
	Let $\alpha: \mf C\to\mf D$ be a functor of Tannakian categories. For every $A,B\in \mf C$ such that $A\subseteq B$, the inclusion $$\alpha_*(H^0_{\mf C}(A))\subseteq H^0_{\mf D}(\alpha(A))\cap \alpha_*(H^0_{\mf C}(B))$$ is an equality. In particular, if $\alpha_*:H^0_{\mf C}(B)\to H^0_{\mf D}(\alpha(B))$ is surjective, then $\alpha_*:H^0_{\mf C}(A)\to H^0_{\mf D}(\alpha(A))$ is surjective as well.
\end{lem}
\begin{proof}
	It is enough to prove the result after extending the scalars of $\mf C$ (see Section~\ref{Tan:n}). In particular, we may assume that $\mf D$ admits a fibre functor $\omega_{\mf D}$. The functor $\omega_{\mf C}:=\omega_{\mf D} \circ \alpha$ is a fibre functor for $\mf C$ because $\alpha$ is exact and faithful by assumption (note that the convention in Section~\ref{tan-fun:ss} is in force). We write $\pi_1(\mf C, \omega_{\mf C})$ and $\pi_1(\mf D, \omega_{\mf D})$ for the automorphism groups of $\omega_{\mf C}$ and $\omega_{\mf D}$ respectively. Let $V$ and $W$ be the tautological representations of $\pi_1(\mf C, \omega_{\mf C})$ on $\omega_{\mf C}(A)$ and $\omega_{\mf C}(B)$ respectively. The functor $\alpha$ induces a morphism $\alpha^*:\pi_1(\mf D, \omega_{\mf D})\to \pi_1(\mf C, \omega_{\mf C})$, which, in turn, induces representations of $\pi_1(\mf D, \omega_{\mf D})$ on the vector spaces $V$ and $W$. The final result then follows from the identity $$H^0(\pi_1(\mf C, \omega_{\mf C}),V)=H^0(\pi_1(\mf D, \omega_{\mf D}),V)\cap H^0(\pi_1(\mf C, \omega_{\mf C}),W)$$ in $H^0(\pi_1(\mf D, \omega_{\mf D}),W).$ 
\end{proof}

\begin{cor}\label{ff-Fable:cor}If $X$ is a smooth variety, the natural functor $\alpha:\oFable(X/K)\to \Isoc(X/K)$ is fully faithful.
\end{cor}
\begin{proof}It is harmless to assume that $X$ is connected, so that both $\oFable(X/K)$ and $\Isoc(X/K)$ are Tannakian categories. In this case, it is enough to prove that for every $\sM^\dag \in \oFable(X/K)$, the morphism 
$$
\alpha_*:H^0_{\oi(X/K)}(\sM^\dag)\to H^0_{\Isoc(X/K)}(\alpha(\sM^\dag))$$ induced by $\alpha$ is an isomorphism. Note that $\alpha_*$ is injective because $\alpha$ is exact. It remains to prove the surjectivity.

\spa By Theorem \ref{F-able-oc:t}, the $F$-able overconvergent isocrystal $\sM^\dag$ embeds into an $F$-invariant overconvergent isocrystal $\sN^\dag$ and, by Lemma \ref{H0:l}, it is enough to show that $$
\alpha_*:H^0_{\oi(X/K)}(\sN^\dag)\to H^0_{\Isoc(X/K)}(\alpha(\sN^\dag))$$
is surjective. Note that the maximal trivial subobject $\sT\subseteq \alpha(\sN^\dag)$ is equal to $H^0_{\Isoc(X/K)}(\alpha(\sN^\dag))\otimes\sO_{X/K}$. Write $\sT^\dag$ for $H^0_{\Isoc(X/K)}(\alpha(\sN^\dag))\otimes\sO_{X/K}^\dag$, so that $\alpha(\sT^\dag)=\sT$. In order to prove that $\alpha_*$ is surjective we have to prove that $\sT^\dag$ is a subobject of $\sN^\dag$. Let $\Phi$ be a Frobenius structure for $\alpha(\sN^\dag)$. By the maximality of $\sT$, the Frobenius structure $\Phi$ preserves $\sT$ so that it induces a Frobenius structure on $\sT$ as well. This induces, in turn, a Frobenius structure to $\sT^\dag$. Thanks to \cite{Ked04}, we have the following identity $$\alpha_*(H^0_{\oi(X/K)}((\sT^\dag)^\vee \otimes \sN^\dag)^{F=\id})=H^0_{\Isoc(X/K)}(\sT^\vee \otimes \alpha(\sN^\dag))^{F=\id}.$$ In particular, the natural Frobenius-equivariant inclusion of convergent isocrystals $\sT \subseteq \alpha(\sN^\dag)$ is induced by a Frobenius-equivariant inclusion of overconvergent isocrystals $\sT^\dag\subseteq \sN^\dag,$ as we wanted. 
\end{proof}

We want to explain now an elementary example where Theorem \ref{F-able-oc:t} fails if we replace overconvergent isocrystals with convergent isocrystals.

\begin{prop}\label{coun-ext:p}
	There exists an extension $$0\to \sO_{X/\Qp}\to \sM \to  \sO_{X/\Qp}\to 0$$ of convergent isocrystals over $\mathbb{A}^1_{\F_p}$, such that $\sM$ is not a subquotient of any $F$-invariant convergent isocrystal.
\end{prop}

\emph{Proof.~}Let $\hat{\mathbb{A}}^1_{\Z_p}$ be the formal affine line over $\Z_p$, endowed with the Frobenius lift $F:\hat{\mathbb{A}}^1_{\Z_p}\to \hat{\mathbb{A}}^1_{\Z_p}$ mapping a coordinate $x$ to $ x^p$. The convergent cohomology groups of $\mathbb{A}^1_{\F_p}$ are equal to the (continuous) de Rham cohomology groups of $\hat{\mathbb{A}}^1_{\Z_p}$. In particular, $$H^1_{\rm conv}(\mathbb{A}^1_{\F_p},  \sO_{X/\Qp})=H^1_{\rm dR}(\hat{\mathbb{A}}^1_{\Z_p}/\Z_p)\otimes \Q_p.$$ Let $\epsilon\in H^1_{\rm dR}(\hat{\mathbb{A}}^1_{\Z_p}/\Z_p)\otimes \Q_p$ be the class represented by the differential form $\sum_{i=0}^{\infty}p^ix^{p^{i^2}}\frac{dx}{x}$. 

\begin{lem}\label{series:l}
The classes $\epsilon_j:=(F^*)^j\epsilon$ are all $\Q_p$-linearly independent when $j$ varies in $\N$.
\end{lem}
\begin{proof}
	Suppose that there exist $c_0,\dots,c_n\in \Q_p$ such that $\sum_{j=0}^n c_j\epsilon_j=0$. This implies that $$\sum_{j=0}^n c_j\sum_{i=0}^{\infty}p^{i+j}x^{p^{i^2+j}}\tfrac{dx}{x}=\sum_{i=0}^{\infty} \sum_{j=0}^n c_jp^{i+j}x^{p^{i^2+j}}\tfrac{dx}{x}=df$$
	where $f= \sum_{\ell=0}^\infty a_\ell x^\ell$ is a series in $\Q_p[[t]]$ with $\lim_{\ell\to +\infty}v_p(a_\ell)=+\infty$. Unravelling this condition, we deduce that $$\lim_{i\to +\infty}(v_p(c_j)+i-i^2)=\lim_{i\to +\infty}v_p(\tfrac{c_jp^{i+j}}{p^{i^2+j}})=\lim_{\ell\to +\infty}v_p(a_\ell)=+\infty$$ for every $0\leq j \leq n$.
	This shows that each $c_j$ is equal to $0$, as we wanted.
\end{proof} 

Let $$0\to \sO_{X/\Qp}\to \sM_{\epsilon} \to  \sO_{X/\Qp}\to 0$$ be the extension of convergent isocrystals over $\mathbb{A}^1_{\F_p}$ induced by $\epsilon$. Suppose, by contradiction, that $\sM_{\epsilon}$ is a subquotient of an $F$-invariant convergent isocrystal $\sN$. Write $\mf{C}$ for the category $\langle\sN \rangle^{\otimes}$. Since $\sN$ is $F$-invariant, $F^*$ restricts to an autoequivalence on $\mf{C}$. This implies that $F^*$ preserves the subspace $$V:=\Ext^1_{\mf{C}}(\sO_{\mathbb{A}^1_{\F_p}/\Q_p},\sO_{\mathbb{A}^1_{\F_p}/\Q_p})\subseteq H^1_{\rm conv}(\mathbb{A}^1_{\F_p}, \sO_{X/\Qp}).$$ The classes $\epsilon_0,\epsilon_1,...$ form an infinite sequence of linearly independent vectors of $V$. This contradicts Lemma \ref{ext-dim:l}.
\qed

\spa There is a variant of Theorem \ref{F-able-oc:t} which works for general convergent isocrystals.
	
\begin{thm}\label{F-inv-iso:t}  If $X$ is a variety over $k$ and  $\sM$ is a subquotient of an $F$-invariant convergent isocrystal $\widetilde{\sM}$, then $\sM$ is  a subobject of an $F$-invariant convergent isocrystal in $\langle\widetilde{\sM}\rangle^\otimes$.
\end{thm}
\begin{proof}Write $\mf C$ for the category $\langle\widetilde{\sM}\rangle^\otimes$. Since $\widetilde{\sM}$ is $F$-invariant, $F^*:\mf C\to \mf C$ is a well-defined autoequivalence. Even in this case, to prove the theorem we make an induction on the length of a Jordan--H\"older filtration of $\sM$. Suppose first that $\sM$ is irreducible. Since $F^*$ permutes the finite set of isomorphims classes of the irreducible subquotients of $\widetilde{\sM}$, there exists $n>0$ such that $(F^*)^n\sM \simeq \sM$. As in Lemma \ref{irr-isoc:l}, we can then embed $\sM$ in the semi-simple $F$-invariant convergent isocrystal $\bigoplus_{i=0}^{n-1} (F^*)^i\sM\in \langle \widetilde{\sM} \rangle^{\otimes}.$ 
	
\spa If $\sM$ is instead an extension of an irreducible convergent isocrystal $\sM_1$ by a convergent isocrystal $\sM_2$, thanks to the previous step, $\sM_1$ is a direct summand of a semi-simple $F$-invariant isocrystal $\sN_1\in\mf C$ and, by the inductive hypothesis, $\sM_2$ is a subobject of some $F$-invariant isocrystal $\sN_2\in\mf C$.  By Proposition \ref{univ-ext:p}, the isocrystal $\sM$ is then a subobject of $\sE_{\mf C}(\sN_2,\sN_1)\oplus \sN_2$, where $\sE_{\mf C}(\sN_2,\sN_1)$ is a universal extenion of $\sN_2$ by $\sN_1$. Thanks to Lemma \ref{ext-dim:l}, the ind-object $\sE_{\mf C}(\sN_2,\sN_1)$ is actually in $\mf C$. As the Frobenius pull-back $F^*$ induces an autoequivalence of $\mf C$, the isocrystal $\sE_{\mf C}(\sN_2,\sN_1)$ is $F$-invariant.
\end{proof}

\subsection{} \label{ss:isoc}
Thanks to the previous results, we are able to recover \cite[Proposition 2.2.4]{AD18}. Fix a characteristic $0$ complete discretely valued field $K$ with residue field $\F_{p^s}$. Let $X$ be a geometrically connected variety over $\F_{p^s}$ and let $x$ be an $\F_{p^s}$-point of $X$. Write $\mf{F^s\textrm{-}}\Isoc(X/K)$ (resp. $\mf{F^s\textrm{-}}\oi(X/K)$) for the category of convergent (resp. overconvergent) isocrystals endowed with an $s$-th Frobenius structure. We have forgetful functors
$\Psi: \mf{F^s\textrm{-}}\Isoc(X/K) \to \Isoc(X/K)$ and $\Psi: \mf{F^s\textrm{-}}\Isoc^{\dag}(X/K) \to \Isoc^{\dag}(X/K)$. Write $\Isoc(X/K)^{\mathrm{geo}}$ (resp. $\oi(X/K)^{\mathrm{geo}}$) for the subcategory of $\Isoc(X/K)$ (resp. $\oi(X/K)$) $\otimes$-generated by the essential image of $\Psi$. All the aforementioned categories are Tannakian categories neutralised by the point $x$. We denote by $\pi_1(-,x)$ the respective Tannaka groups.
\begin{cor}	\label{conv-oconv-mon:c}
			 There exists a functorial commutative diagram 
			\begin{center}
				\begin{tikzcd}
					1\arrow{r} & \pi_1(\Isoc(X/K)^{\mathrm{geo}},x)\arrow[r,"\Psi^*"]\arrow[d] & \pi_1(\mf{F^s\textrm{-}}\Isoc(X/K),x)\arrow[r, "p_{X*}"]\arrow[d] & \pi_1(\mf{F^s\textrm{-}}\Isoc(\F_{p^s}))\arrow[r]\arrow[d,"="] &1\\
					1\arrow{r} & \pi_1(\oi(X/K)^{\mathrm{geo}},x)\arrow[r,"\Psi^*"] & \pi_1(\mf{F^s\textrm{-}}\oi(X/K),x)\arrow[r,"p_{X*}"] & \pi_1(\mf{F^s\textrm{-}}\Isoc(\F_{p^s}))\arrow{r} & 1
				\end{tikzcd}
			\end{center}
		with exact rows.
	\end{cor}
\begin{proof}
	The vertical arrows of the diagram are constructed using the functor $\alpha:\oi(X)\to \Isoc(X)$. To prove that the right horizontal morphisms $p_{X*}$ are faithfully flat, we note that they admit a section induced by the $\F_{p^s}$-point $x$. For the exactness in the middle we want to apply instead Proposition \ref{a-prop:exseq}. Thanks to Theorem \ref{F-able-oc:t} and Theorem \ref{F-inv-iso:t}, the functor $\Psi$ is observable both for convergent and overconvergent isocrystals. To conclude, it is enough to note that if $(\sM,\Phi)$ is an $F$-isocrystal (both convergent or overconvergent) the maximal trivial subobject of $\sM$ is preserved by $\Phi$, thus it admits a Frobenius structure compatible with $\Phi$.
\end{proof}

\appendix
\section{Observable functors} \label{sec:tannaka}
\subsection{}In \cite{BBHM63}, the authors introduce the notion of observable subgroup, which turns out to be useful to study Tannaka groups. Indeed, it appeared implicitly and in various forms in some works on Tannakian categories (e.g. \cite{EHS07}, \cite{dS15}, and \cite{AE19}). Recently in \cite{And20}, André studies the interaction of this notion with other standard results on Tannaka groups. We propose in the sequel a self-contained variant of André's approach. We introduce the more general notion of \textit{\textit{observable functor}} and use it to provide new criteria for the surjectivity and exactness of morphisms between Tannaka groups, and to give new proofs of previous results. 

\begin{defn}
Let $k$ be a field. We say that a functor $\alpha: \mf C\to \mf D$ of $k$-linear Tannakian categories is \textit{observable} if for every rank $1$ object $L\in \mf D$ which embeds into an object of $\alpha(\mf C)$, there exists $n>0$ such that $(L^{\otimes n})^\vee$ embeds into an object of $\alpha(\mf C)$.
\end{defn}

\begin{prop}[Bialynicki-Birula--Hochschild--Mostow]\label{a-observable:p}
A functor $\alpha: \mf C\to \mf D$ of Tannakian categories is observable if and only if every quotient in $\mf D$ of an object in $\alpha(\mf C)$ embeds into an object of $\alpha(\mf C)$.
\end{prop}

\begin{proof}A stronger version of this theorem is proven in \cite[Thm. 9]{BBHM63}. We give here a shorter proof of  this weaker version. It is clear that the second condition implies the first one, thus it is enough to prove the converse. Suppose given an exact sequence $$0\to W\to V \to Q \to 0 $$in $\mf D$ with $V\in\alpha(\mf C)$. Write $r$ for the rank of $W$ and $L$ for $\wedge^r(W)$. Since $L$ embeds into $\wedge^r(V)\in\alpha(\mf C)$, there exists $n>0$ and $N\in \alpha(\mf C)$ such that $(L^{\otimes n})^\vee$ is a subobject of $N$. At the same time, the natural morphism $V\otimes L\to \wedge^{r+1}(V)$ factors through $Q\otimes L$. Combining these two observations, we deduce that there exists an embedding $$Q=(Q\otimes L)\otimes L^{\otimes (n-1)}\otimes (L^{\otimes n})^\vee\hookrightarrow  \wedge^{r+1}(V)\otimes (\wedge^{r}(V))^{\otimes(n-1)}\otimes N.$$
This shows that $Q$ embeds into an object of $\alpha(\mf C)$, as we wanted.
\end{proof}

\begin{lem}\label{a-suf-obs:l}
The functor $\alpha$ is observable if one of the following conditions is satisfied.
\begin{itemize}
	\item[1)]If a rank $1$ object $L\in \mf D$ embeds into an object of $\alpha(\mf C)$, then it is the direct summand of a semi-simple object of $\alpha(\mf C)$.
	\item[2)] The functor $\alpha$ sends semi-simple objects to semi-simple objects.
\end{itemize}
\end{lem}
\begin{proof}
Suppose that $\alpha$ satisfies 1). If $L\in \mf D$ is a rank 1 subobject of $V\in \alpha(\mf C)$, then there exists a semi-simple object $W\in \alpha(\mf C)$ such that $L$ is a direct summand of $W$. Therefore, $L^\vee$ is a direct summand of $W^\vee\in \alpha(\mf C)$. This shows that $\alpha$ is observable.

\medskip

We prove now that 2) implies 1). Let $M$ be an object of $\mf C$ and $L$ be a rank $1$ subobject of $\alpha(M)$. Choose an ascending filtration $F_\bullet$ of $M$ such that the associated graded object ${\rm Gr}_{F_\bullet}(M)$ is semi-simple. Since $L$ is of rank $1$, there exists $i$ such that $L$ embeds into $\alpha(F_i(M))/\alpha(F_{i-1}(M))$. Note that $\alpha(F_i(M))/\alpha(F_{i-1}(M))=\alpha(F_i(M)/F_{i-1}(M))$ because $\alpha$ is exact and $\alpha(F_i(M)/F_{i-1}(M))$ is semi-simple by 2). We deduce that $L$ embeds into a semi-simple object, as we wanted. 
\end{proof}

\subsection{} We are now ready to see how the notion of observable functor interacts with various properties of morphisms of Tannaka groups. Let 
\ga{}{ K\xrightarrow{f} G \xrightarrow{g} H \notag}
be morphisms of affine groups schemes over $k$. 
We denote by 
\ga{}{ \mf H  \xrightarrow{g^*} \mf G \xrightarrow{f^*} \mf K \notag}
the induced functors between the representation categories with values in the category of finite-dimensional $k$-vector spaces. If $V$ is a representation in $\mf H$, we write $\mathbb{P}(V)$ for the projectivisation of $V$, endowed with the natural action of $H$.

\begin{prop}[Faithful flatness] \label{a-prop:ff}
	The following conditions are equivalent.
	\begin{itemize}
		\item[1)]
		$g$ is faithfully flat.
		
		\item[2)]  $g^*(\mf H)\subseteq \mf G$ is a full subcategory stable under the operation of taking subobjects.
		
		\item[3)]  $g^*$ is fully faithful and observable.
		
		\item[4)] For every $V\in \mf H$, the inclusion $\mathbb{P}(V)^H(k)\subseteq \mathbb{P}(V)^G(k)$ is an equality.
		
		\end{itemize}
	
\end{prop} 
\begin{proof}
	The equivalence $1) \Leftrightarrow 2) $ is \cite[Prop. 2.21 (a)]{DM82} and the implications $1)\Rightarrow 4)$ and $2)\Rightarrow 3)$ are immediate. We prove $3)\Rightarrow 2).$ Let us assume that $g^*$ is fully faithful and observable. For every $V\in \mf G$ which is a subobject of some $M\in g^*(\mf H)$, there exists, by Proposition \ref{a-observable:p}, an $N\in g^*(\mf H)$ which has $V^\vee$ as a subobject. Dualising we deduce that there exists a surjective morphism $N^{\vee}\twoheadrightarrow V$. This implies that there exists a morphism $a:N^{\vee}\to M$ in $\mf G$ such that ${\rm Im}(a)=V$. Since $g^*$ is fully faithful, $a=g^*(\widetilde{a})$ for some morphism $\widetilde{a}\in \Hom_{\mf H}(N^\vee,M)$. We deduce that $V=g^*({\rm Im}(\widetilde{a}))$, as we wanted.
	
	\medskip
	
	Finally, we prove $4)\Rightarrow 3)$. Suppose that $g^*$ satisfies 4). This implies that $g^*$ is observable because every rank 1 subobject in $\mf G$ which embeds into an object of $g^*(\mf H)$ is itself in $g^*(\mf H)$. The next step is to prove that $g^*$ induces an injective morphism on the groups of characters $g^*:X^*(H)\to X^*(G)$. To do this we observe that if $L_1$ and $L_2$ are two rank $1$ objects in $\mf H$, then $L_1\simeq L_2$ if and only if $|\mathbb{P}(V)^H(k)|>2$, where $V:=L_1\oplus L_2$. Of course, the same criterion works also for the rank $1$ objects of $\mf G$. This shows that $L_1\simeq L_2$ if and only if $g^*L_1\simeq g^*L_2$, because, by condition 4), the sets $\mathbb{P}(V)^H(k)$ and $\mathbb{P}(g^*V)^H(k)$ are equal. Using this we can finally prove that $g^*$ is fully faithful. Indeed, thanks to 4), for every $V\in \mf H$, the maximal trivial subobject $T\subseteq g^*(V)$ comes from some $L^{\oplus r}\subseteq V$. By the injectivity of the morphism on the groups of characters, we deduce that $L$ is isomorphic to the unit object. This yields the desired result.
\end{proof}

	\begin{rmk}The equivalence $1)\Leftrightarrow 4)$  appears in \cite[Lem. 4.2 and Lem. 4.3]{dS15}. Also, a variant of $3)$ is considered in \cite[Lem.~1.6]{AE19}. We correct a typo in {\it loc. cit.}: $\lambda/m \in \bar \Q_p $ should read $\sqrt[m]{\lambda} \in \bar \Q_p.$

 \end{rmk}

\begin{prop}[Closed immersion] \label{a-prop:clemb}
The following conditions are equivalent.
\begin{itemize}
\item[1)]
$f$ is a closed immersion.
\item[2)] Every object of $\mf K$ is a subquotient of an object coming from $\mf G$.
\end{itemize}
\end{prop}
\begin{proof}
This is \cite[Prop.~2.21 (b)]{DM82}.
\end{proof}

\begin{defn}
We say that a closed subgroup $f:K\hookrightarrow G$ such that $f^*$ is an observable functor is an \textit{observable subgroup}.
\end{defn}
\begin{cor}\label{a-sub:c}
The following conditions are equivalent.
\begin{itemize}
	\item[1)] $f:K\hookrightarrow G$ is an observable subgroup.
	\item[2)] Every object of $\mf K$ embeds into an object coming from $\mf G$.
\end{itemize}
\end{cor}
\begin{proof}The implication $1)\Rightarrow 2)$ follows from Proposition \ref{a-prop:clemb} and from Proposition \ref{a-observable:p}  while the converse follows from Proposition \ref{a-prop:clemb}.
\end{proof}

\begin{lem}\label{a-comp-obs:l} Suppose that $f:K\to G$ decomposes as $K\xrightarrow{f_1} K'\xrightarrow{f_2} G$ with $K'$ an affine group scheme and $f_2$ a closed immersion. If $f^*$ is observable, then $f_1^*$ is observable.
\end{lem}
\begin{proof}Write $\mf{K}'$ for the category of representations of $K'$. Suppose that $V\in \mf K$ embeds into some representation coming from $\mf K'$, we want to show that the same is true for $V^\vee$. Since $f_2$ is a closed immersion, by Proposition \ref{a-prop:clemb} every representation in $\mf{K}'$ is a subquotient of a representation coming from $\mf G$. Thus $V^\vee$ is a subquotient of a representation coming from $\mf G$. Since $f^*$ is observable, by Proposition \ref{a-observable:p}, the representation $V^\vee$ embeds then into a representation coming from $\mf{G}$. This shows that $f_1^*$ is observable.
\end{proof}
\begin{prop}[Normality] \label{a-prop:norm}
		The following conditions are equivalent.
	\begin{itemize}
		\item[1)]
		$f(K)\subseteq G$ is a closed normal subgroup.
		\item[2)] $f^*$ is observable and for every $M\in\mf G$, the maximal trivial subobject of $f^*M$ comes from a subobject of $M$ in $\mf G$.
	\end{itemize}
\end{prop}
\begin{proof}
To prove $1)\Rightarrow 2)$, we note that by Clifford's theorem $f^*$ sends semi-simple objects to semi-simple objects. By Lemma \ref{a-suf-obs:l}, this implies that $f^*$ is observable. For the second part of condition 2) it is enough to observe that since $f(K)$ is normal in $G$, the maximal trivial subobject of $f^*M$ is stable under the action of $G$.

\spa
We prove now $2)\Rightarrow 1)$. Let $K'$ be the smallest closed normal subgroup of $G$ containing the image of $f$ and $\mf K'$ the induced category of representations. Then $f$ decomposes as $$K\xrightarrow{f_1} K'\xrightarrow{f_2} G,$$ with $f_2$ a closed immersion. Since $f^*$ is observable, thanks to Lemma \ref{a-comp-obs:l}, the same is true for $f_1^*$. By Proposition \ref{a-prop:ff}, in order to prove that $f_1$ is faithfully flat, and thus that $K'=f(K)$, it remains to prove that $f_1^*$ is fully faithful.

\spa

 We want to apply Lemma \ref{H0:l}. Since $K'$ is an observable subgroup of $G$, we know that every $V\in \mf{K}'$ embeds into an object $W=f_2^*M$ with $M\in \mf G$. By Condition 2), there exists a subobject $N\subseteq M$ which is sent by $f^*$ to the maximal trivial subobject $T\subseteq f^*M=f_1^*W$. By the construction of $K'$, the subobject $f_2^*N\subseteq f_2^*M=W$, which is sent to $T\subseteq f^*M$, is trivial in $\mf K'$. This shows that the map $f_{1}^*: H^0_{\mf{K}'}(W)\to H^0_{\mf{K}}(f^*_1(W))$ is an isomorphism. We end the proof thanks to Lemma \ref{H0:l}.
\end{proof}

\begin{prop}[Exact sequence] \label{a-prop:exseq}  Assume that $f$ is a closed immersion, $g$ is faithfully flat and $g\circ f$ is trivial.
The following conditions are equivalent.
\begin{itemize}
\item[1)]  $1\to K\xrightarrow{f} G \xrightarrow{g} H\to 1$ is an exact sequence.
\item[2)] $f^*$ is observable and for every $M\in\mf G$, there exists $U\in \mf H$ such that the maximal trivial subobject of $f^*M$ comes from $g^*(U)\subseteq M$.

\end{itemize}

\end{prop}
\begin{proof}
This can be proven combining \cite[Thm.~A.1]{EHS07} and Corollary \ref{a-sub:c}. We present here a variant of this proof. The implication $1)\Rightarrow 2)$ follows from Proposition \ref{a-prop:norm}. For the converse, we know by Proposition \ref{a-prop:norm} that the closed subgroup $f:K\hookrightarrow G$ is normal, thus it remains to prove that the quotient morphism $G/K\twoheadrightarrow H$ is a closed immersion. By Proposition \ref{a-prop:clemb}, this amounts to showing that every representation of $G$ which is trivial when restricted to $K$ comes from a representation of $H$. This is guaranteed by Condition 2).

\end{proof}

\begin{rmk}Proposition \ref{a-prop:norm} corresponds to \cite[Prop. C.3]{And20}, while combining Lemma \ref{a-suf-obs:l} and Proposition \ref{a-prop:exseq} one recovers \cite[Prop. 2.6]{LP17}.
\end{rmk}

\bibliographystyle{alpha}

\end{document}